\documentclass[a4paper]{article}

\usepackage{amsfonts, amsmath, amssymb, amsthm}
\usepackage{graphicx}
\usepackage{fullpage}
\usepackage{float}
\usepackage{enumitem}

\newtheorem{theorem}{Theorem}[section]
\newtheorem{corollary}{Corollary}[theorem]
\newtheorem{conjecture}{Conjecture}[section]

\newtheorem{definition}{Definition}

\newcommand{\Z}{\mathbb{Z}}
\newcommand{\N}{\mathbb{N}}

\renewcommand{\dim}[1]{\text{dim}\left( #1 \right)}

\begin{document}
	\title{Resolution to Sutner's Conjecture}
	\author{William Boyles\thanks{Department of Mathematics, North Carolina State University, Raleigh, NC 27695 (wmboyle2@ncsu.edu)}}
	\date{\today}
	\maketitle
	
	\section{Introduction}
	Consider a game played on a simple graph $G = (V,E)$ where each vertex consists of a clickable light.
	Clicking any vertex $v$ toggles on on/off state of $v$ and its neighbors.
	One wins the game by finding a sequence of clicks that turns off all the lights.
	When $G$ is a $5 \times 5$ grid, this game was commercially available from Tiger Electronics as \textit{Lights Out}.
	
	Sutner was one of the first to study these games mathematically.
	He showed that for any $G$ the initial configuration of all lights on is solvable \cite{Sutner1989}.
	He also found that when $d(G) = \dim{\ker{(A + I)}}$ over the field $\Z_2$, where $A$ is the adjacency matrix of $G$, is 0 all initial configurations are solvable.
	In particular, 1 out of every $2^{d(G)}$ initial configurations are solvable, while each solvable configuration has $2^{d(G)}$ distinct solutions \cite{Sutner1989}.
	When investigating $n \times n$ grid graphs, Sutner conjectured the following relationship:
	\begin{align*}
		d_{2n+1} &= 2d_n + \delta_n \text{, } \delta_n \in \{0,2\} \\
		\delta_{2n+1} &= \delta_n,
	\end{align*}
	where $d_n = d(G)$ for $G$ an $n \times n$ grid graph \cite{Sutner1989}.
	
	We resolve this conjecture in the affirmative.
	We use results from Sutner that give the nullity of a $n \times n$ board as the GCD of two polynomials in the ring $\Z_2[x]$ \cite{Sutner96sigma-automataand}.
	We then apply identities from Hunziker, Machiavelo, and Park that relate the polynomials $(2n+1) \times (2n+1)$ grids and $n \times n$ grids \cite{HUNZIKER2004465}.
	We then apply a result from Ore about the GCD of two products \cite{ore_number_theory}.
	Together, these results allow us to prove Sutner's conjecture.
	We then go further and show for exactly which values of $n$ $\delta_n$ is 0 or 2.
	
	\section{Fibonacci Polynomials}
	\begin{definition}
		Let $f_n(x)$ be the polynomial in the ring $\Z_2[x]$ defined recursively by
		\begin{equation*}
			f_n(x) = \begin{cases}
				0 & n=0 \\
				1 & n=1 \\
				xf_{n-1}(x) + f_{n-2}(x) & \text{otherwise }.
			\end{cases}
		\end{equation*}
	\end{definition}
	These polynomials are often referred to as Fibonacci polynomials because when defined over the reals, evaluating $f_n(x)$ at $x=1$ gives the $n$th Fibonacci number.
	Since coefficients of $f_n(x)$ are either 1 or 0, one can visualize them by coloring squares black or white to represent the coefficients.
	For example,
	\begin{equation*}
		f(6,x) = 1x^5 + 0x^4 + 0x^3 + 0x^2 + 1x + 0 = \blacksquare\square\square\square\square\blacksquare\square.
	\end{equation*}
	Plotting $f(1,x)$ through $f(128,x)$ in the same way, aligning terms of the same degree, we see a Sierpinski triangle rotated 90 degrees.
	
	\begin{figure}[H]
		\centering
		\includegraphics[width=.8\textwidth]{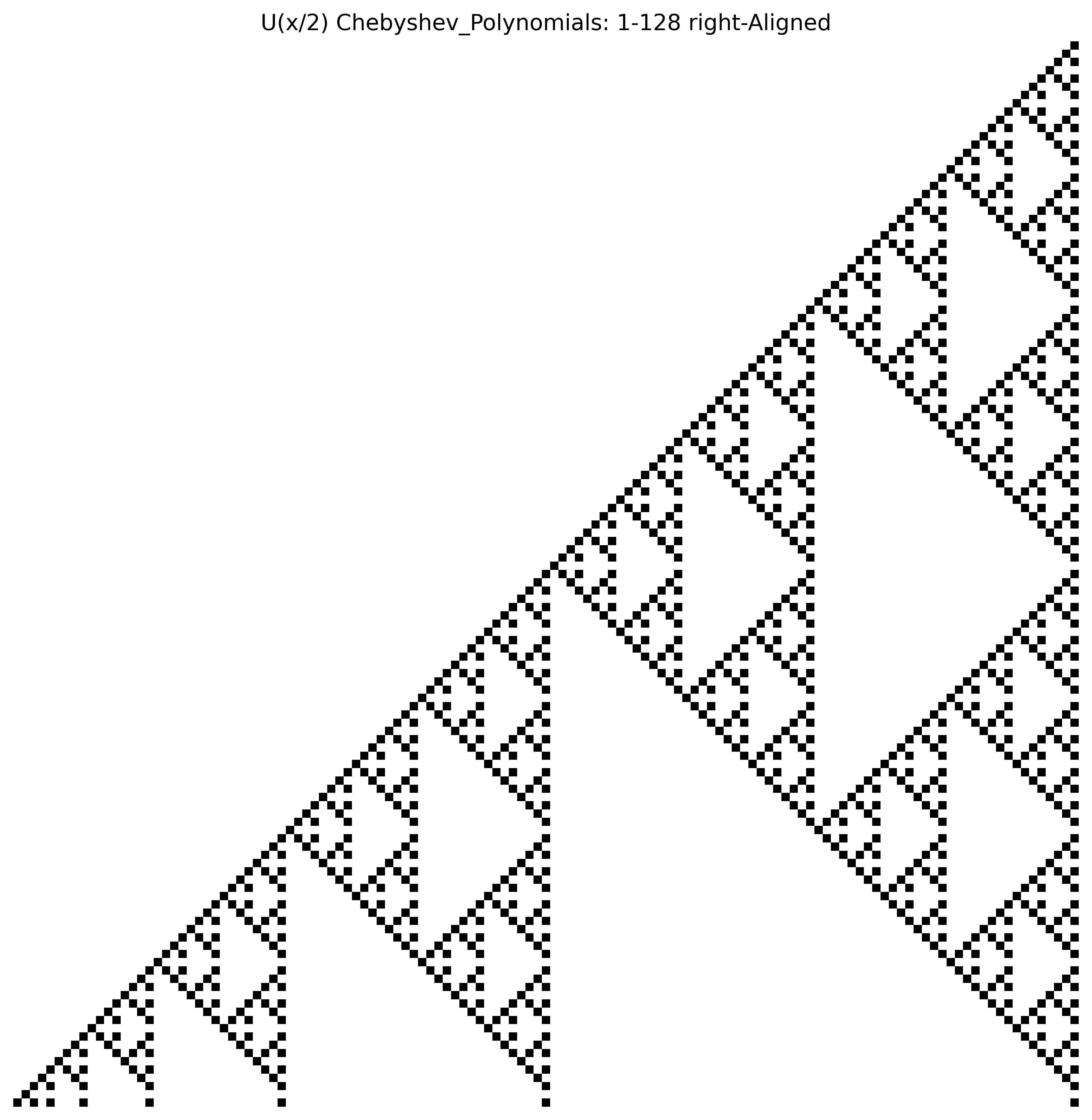}	
	\end{figure}
	
	Sutner, using a well-known connection between the Sierpinski triangle and parity of binomial coefficients notes \cite{Sutner96sigma-automataand}
	\begin{equation*}
		f_n(x) = \sum_{i}{\binom{n+i}{2i+1}x^i \mod 2}.
	\end{equation*}

	Sutner showed how to calculate $d_n$ as the degree of the GCD of two polynomials in $\Z_2[x]$ \cite{Sutner96sigma-automataand}.
	\begin{theorem}[Sutner]\label{Sutner_gcd}
		For all $n \in \N$.
		\begin{equation*}
			d_n = \deg{\gcd\left(f_{n+1}(x), f_{n+1}(x+1)\right)}.
		\end{equation*}
	\end{theorem}
	
	The recursive definition given in Theorem \ref{Sutner_gcd} and Sutner's formula in terns of the parity of binomial coefficients provide brute force ways to calculate $f_n(x)$.
	Hunziker, Machiavelo, and Park show the following identity that makes calculating $f_n(x)$ easier when $n$ is divisible by powers of 2 \cite{HUNZIKER2004465}.
	
	\begin{theorem}[Hunziker, Machiavelo, and Park]\label{HMP_identity}
		Let $n = b\cdot2^{k}$ where $b$ and $k$ are non-negative integers.
		Then
		\begin{equation*}
			f_n(x) = x^{2^{k}-1} f_{b}^{2^{k}}(x).
		\end{equation*}
	\end{theorem}

	In particular, we will use this result to relate $f_{2n+2}(x)$ and $f_{4n+4}(x)$ to $f_{n+1}(x)$.
	\begin{corollary}\label{cor1}
		The following identities hold:
		\begin{align*}
			f_{2n+2}(x) &= xf_{n+1}^2(x) \\
			f_{4n+4}(x) &= x^3f_{n+1}^4(x).
		\end{align*}
	\end{corollary}
	\begin{proof}
		Notice that $2n+2 = (n+1)2^1$ and $4n+4 = (n+1)2^{2}$.
		Thus, our desired identities follow from Theorem \ref{HMP_identity}.
	\end{proof}

	Now that we have a way to express $f_{2n+2}(x)$ and $f_{4n+4}(x)$ as a product of $f_{n+1}(x)$ and a power of $x$, we need a way to express the GCD of products so we can relate $d_{2n+1}$ and $d_n$.
	This is exactly what a number-theoretic result from Ore provides \cite{ore_number_theory}.
	
	\begin{theorem}[Ore]\label{Ore_gcd}
		Let $a$, $b$, $c$, and $d$ be integers.
		Let $(a,b)$ denote $\gcd{(a,b)}$.
		Then
		\begin{equation*}
			(ab,cd) = (a,c)(b,d)\left(\frac{a}{(a,c)},\frac{d}{(b,d)}\right)\left(\frac{c}{(a,c)},\frac{b}{(b,d)}\right).
		\end{equation*}
	\end{theorem}

	Although Ore's result deals specifically with integers, both the integers and $\Z_2[x]$ are Euclidean domains, so the result will still hold.
	
	Hunziker, Machiavelo, and Park also showed the following identity \cite{HUNZIKER2004465}.
	\begin{theorem}[Hunziker, Machiavelo, and Park]\label{HMP_gcd}
		A polynomial $\tau(x)$ in $\Z_2[x]$ divides both $f_n(x)$ and $f_m(x)$ if and only if it divides $f_{\gcd{(m,n)}}$.
		In particular,
		\begin{equation*}
			\gcd{\left(f_m(x), f_n(x)\right)} = f_{\gcd{(m,n)}}(x).
		\end{equation*}
	\end{theorem}

	We specifically will use the following corollary:
	\begin{corollary}\label{cor_HMP_gcd}
		For some polynomial $\tau(x)$ in $\Z_2[x]$, let $n \geq 0$ be the smallest integer such that $\tau(x)$ divides $f_n(x)$.
		Then for all $m \geq 0$, $\tau(x)$ divides $f_m(x)$ if and only if $n$ divides $m$.
	\end{corollary}
	\begin{proof}
		Let $\tau(x)$ be some polynomial in $\Z_2[x]$.
		Let $f_n(x)$ be the smallest Fibonacci polynomial such that $\tau(x)$ divides $f_n(x)$.
		
		Assume that $\tau(x)$ divides $f_m(x)$ for some number $m$.
		Then $\tau(x)$ is a common factor of $f_m(x)$ and $f_n(x)$, so Theorem \ref{HMP_gcd} tells us that $\tau(x)$ divides $f_{\gcd{(m,n)}}(x)$.
		Since $f_n(x)$ is the smallest Fibonacci polynomial that is divisible by $\tau(x)$,
		\begin{equation*}
			\gcd{(m,n)} \geq n.
		\end{equation*}
		This inequality only holds if $\gcd{(m,n)} = n$.
		Thus, $m$ must be a multiple of $n$ as desired.
		
		Now assume that $m$ is a multiple of $n$.
		Then $\gcd{(m,n)} = n$.
		Theorem \ref{HMP_gcd} tells us
		\begin{equation*}
			\gcd{\left(f_m(x), f_n(x)\right)} = f_{\gcd{(m,n)}}(x) = f_n(x).
		\end{equation*}
		Since $\tau(x)$ divides $f_n(x)$, and $f_n(x)$ is the GCD of $f_m(x)$ and $f_n(x)$, $\tau(x)$ must also divide $f_m(x)$, as desired.
	\end{proof}

	In particular, we will use the following instances of Corollary \ref{cor_HMP_gcd} to determine when $\delta_n$ is 0 or 2.
	
	\begin{corollary}\label{cor_cor_HMP_gcd}
		The following are true:
		\begin{enumerate}[label=(\roman*)]
			\item $x$ divides $f_n(x)$ if and only if $n \equiv 0 \mod 2$.
			\item $x+1$ divides $f_n(x+1)$ if and only if $n \equiv 0 \mod 2$.
			\item $x+1$ divides $f_n(x)$ if and only if $n \equiv 0 \mod 3$.
			\item $x$ divides $f_n(x+1)$ if and only if $n \equiv 0 \mod 3$.
		\end{enumerate}
	\end{corollary}
	\begin{proof}
		Notice,
		\begin{enumerate}[label=(\roman*)]
			\item
				We find that $f_2(x) = x$ is the smallest Fibonacci polynomial divisible by $x$, so we apply Corollary \ref{cor_HMP_gcd} to get the desired result.
			\item
				Follows from (i) by substituting $x+1$ for $x$.
			\item
				We find that $f_3(x) = (x+1)^2$ is the smallest Fibonacci polynomial divisible by $x+1$, so we apply Corollary \ref{cor_HMP_gcd} to get the desired result.
			\item
				Follows from (iii) by substituting $x+1$ for $x$.
		\end{enumerate}
	\end{proof}

	\section{Proof of Sutner's Conjecture}
	Finally, we are ready to prove Sutner's conjecture \cite{Sutner1989}. 
	\begin{theorem}\label{sutners-thm}
		For all $n \in \N$,
		\begin{equation*}
			d_{2n+1} = 2d_n + \delta_n,
		\end{equation*}
		where $\delta_n \in \{0,2\}$, and $\delta_{2n+1} = \delta_n$.
	\end{theorem}
	\begin{proof}
		Let $(a,b)$ denote $\gcd{(a,b)}$.
		Applying the results from Theorems \ref{Sutner_gcd}, \ref{HMP_identity}, and \ref{Ore_gcd},
		\begin{align*}
			d_{2n+1} &= \deg \left(f_{2n+2}(x), f_{2n+2}(x+1)\right) \\
				&= \deg \left(xf^2_{n+1}(x), (x+1)f^2_{n+1}(x+1)\right) \\
				&= \deg (x,x+1) \left(f^2_{n+1}(x),f^2_{n+1}(x+1)\right) \left(\frac{x+1}{(x,x+1)},\frac{f^2_{n+1}(x)}{(f^2_{n+1}(x),f^2_{n+1}(x+1))}\right) \left(\frac{x}{(x,x+1)},\frac{f^2_{n+1}(x+1)}{(f^2_{n+1}(x),f^2_{n+1}(x+1))}\right) \\
				&= \deg \left(f_{n+1}(x),f_{n+1}(x+1)\right)^2 \left(x+1,\frac{f^2_{n+1}(x)}{(f_{n+1}(x),f_{n+1}(x+1))^2}\right) \left(x,\frac{f^2_{n+1}(x+1)}{(f_{n+1}(x),f_{n+1}(x+1))^2}\right) \\
				&= \deg \left(f_{n+1}(x),f_{n+1}(x+1)\right)^2 \left(x+1,\frac{f_{n+1}(x)}{(f_{n+1}(x),f_{n+1}(x+1))}\right) \left(x,\frac{f_{n+1}(x+1)}{(f_{n+1}(x),f_{n+1}(x+1))}\right) \\
				&= 2d_n + \deg\left(x+1,\frac{f_{n+1}(x)}{(f_{n+1}(x),f_{n+1}(x+1))}\right) \left(x,\frac{f_{n+1}(x+1)}{(f_{n+1}(x),f_{n+1}(x+1))}\right).
		\end{align*}
		Notice that if we substitute $x+1$ for $x$,
		\begin{equation*}
			\left(x+1,\frac{f_{n+1}(x)}{(f_{n+1}(x+1),f_{n+1}(x))}\right) \text{ becomes } \left(x,\frac{f_{n+1}(x+1)}{(f_{n+1}(x),f_{n+1}(x+1))}\right).
		\end{equation*}
		Thus, we see that these two remaining GCD terms in our expression for $d_{2n+1}$ are either both 1 or not 1 simultaneously.
		This means we can further simplify to
		\begin{equation*}
			d_{2n+1} = 2d_{n} + 2\deg \left(x,\frac{f_{n+1}(x+1)}{(f_{n+1}(x),f_{n+1}(x+1))}\right).
		\end{equation*}
		So, we see that
		\begin{equation*}
			d_{2n+1} = 2d_{n} + \delta_n \text{, where }\delta_n = 2\deg \left(x,\frac{f_{n+1}(x+1)}{(f_{n+1}(x),f_{n+1}(x+1))}\right).
		\end{equation*}
		Thus, $\delta_n \in \{0,2\}$ as desired.
		
		What remains is to show that $\delta_{n} = \delta_{2n+1}$.
		Applying Corollary \ref{cor1},
		\begin{align*}
			d_{4n+3} &= \deg \left(x^3f^4_{n+1}(x),(x+1)^3f^4_{n+1}(x+1)\right) \\
				&= \deg \left(x^3,(x+1)^3\right) \left(f^4_{n+1}(x),f^4_{n+1}(x+1)\right) \left(x^3,\frac{f^4_{n+1}(x+1)}{(f^4_{n+1}(x),f^4_{n+1}(x+1))}\right) \left((x+1)^3,\frac{f^4_{n+1}(x)}{(f^4_{n+1}(x),f^4_{n+1}(x+1))}\right) \\
				&= \deg \left(f_{n+1}(x),f_{n+1}(x+1)\right)^4 \left(x^3,\frac{f^4_{n+1}(x+1)}{(f_{n+1}(x),f_{n+1}(x+1))^4}\right) \left((x+1)^3,\frac{f^4_{n+1}(x)}{(f_{n+1}(x),f_{n+1}(x+1))^4}\right) \\
				&= \deg \left(f_{n+1}(x),f_{n+1}(x+1)\right)^4 \left(x^3,\frac{f^3_{n+1}(x+1)}{(f_{n+1}(x),f_{n+1}(x+1))^3}\right) \left((x+1)^3,\frac{f^3_{n+1}(x)}{(f_{n+1}(x),f_{n+1}(x+1))^3}\right) \\
				&= \deg \left(f_{n+1}(x),f_{n+1}(x+1)\right)^4 \left(x,\frac{f_{n+1}(x+1)}{(f_{n+1}(x),f_{n+1}(x+1))}\right)^3 \left(x+1,\frac{f_{n+1}(x)}{(f_{n+1}(x),f_{n+1}(x+1))}\right)^3 \\
				&= 4d_n + 3\delta_n.
		\end{align*}
		Also, from our work previously in this proof,
		\begin{align*}
			d_{4n+3} &= d_{2(2n+1) + 1} \\
				&= 2 d_{2n+1} + \delta_{2n+1} \\
				&= 2 \left(2d_{n} + \delta_{n}\right) + \delta_{2n+1} \\
				&= 4d_{n} + 2\delta_{n} + \delta_{2n+1}.
		\end{align*}
		For these two expressions for $d_{4n+3}$ to be equal, we must have $\delta_{2n+1} = \delta_n$, as desired.
	\end{proof}

	This result may have been proven prior by Yamagishi \cite{YAMAGISHI20151}.
	However, Yamagishi does not mention the connection to Sutner's conjecture, and the proof provided is not as direct as the one we provide because Yamagishi's work is concerned with tori rather than grids.
	
	\section{When does $\delta_n = 2$?}
	Theorem \ref{sutners-thm} proves Sutner's conjecture as stated and even gives a formula for finding $\delta_n$.
	However, this formula is somewhat messy, containing one polynomial division and two polynomial GCDs.
	We can improve this formula to just a modulo operation on $n$.
	We'll do so by using the divisibility properties established in Corollary \ref{cor_cor_HMP_gcd}.
	
	\begin{theorem}\label{when-is-deltan-2}
		The value of $\delta_n$ is 2 if and only if $n+1$ is divisible by 3.
	\end{theorem}
	\begin{proof}
		From our work in Theorem \ref{sutners-thm}, we know that
		\begin{equation*}
			\delta_n = 2\deg\left(x+1,\frac{f_{n+1}(x)}{(f_{n+1}(x),f_{n+1}(x+1))}\right).
		\end{equation*}
		So we see that $\delta_n$ is 2 exactly when $f_{n+1}(x)$ can be divided without remainder by $x+1$ more times than $f_{n+1}(x+1)$.
		
		For $n+1$ is not divisible by 3, Corollary \ref{cor_cor_HMP_gcd} tells us that $f_{n+1}(x)$ is not divisible by $x+1$.
		So in this case, $\delta_n = 0$, as desired.
		
		For $n+1$ divisible by 3, let $n+1 = b \cdot 2^k$ for some integers $b, k \geq 0$ where $b$ is odd.
		Since $n+1$ is divisible by 3, $b$ must also be divisible by 3.
		Applying Corollary \ref{cor1},
		\begin{equation*}
			f_{n+1}(x) = x^{2^k - 1}f_{b}^{2^k}(x) \text{ and } f_{n+1}(x+1) = (x+1)^{2^k - 1}f_{b}^{2^k}(x+1).
		\end{equation*}
		Since $b$ is an odd multiple of 3, Corollary \ref{cor_cor_HMP_gcd} tell us that $x+1$ divides $f_b(x)$, but $x+1$ does not divide $f_b(x+1)$.
		So,
		\begin{equation*}
			f_{n+1}(x) = x^{2^k - 1}(x+1)^{2^k}g^{2^k}(x) \text{ and } f_{n+1}(x+1) = (x+1)^{2^k - 1}x^{2^k}g^{2^k}(x+1),
		\end{equation*} 
		for some $g(x) \in \Z_2[x]$, where $g(x)$ and $g(x+1)$ are both divisible by neither $x$ nor $x+1$.
		So, we see that $f_{n+1}(x)$ can be divided without remainder by $x+1$ one more time than $f_{n+1}(x+1)$.
		So, $\delta_n = 2$, as desired.
	\end{proof}

	\section{Future Work}
	There are many other relationships with $d_n$, many yet to be proven.
	For example, Sutner mentions that for all $k \in \N$, $d_{2^k - 1} = 0$ \cite{Sutner1989}.
	We believe that the following relationships hold, but are unaware of a proof.
	
	\begin{conjecture}\label{conj-all-2}
		There are infinitely many $n$ such that $d_n = 2$.
		In particular, for all $k \in \N$, $d_{2\cdot 3^{k} - 1} = 2$.
	\end{conjecture}
	This conjecture is similar to Sutner's result that shows there are infinitely many $n$ such that $d_n = 0$.
	
	\begin{conjecture}\label{conj-powers}
		Let $a$ be an odd natural number.
		If $a$ is not divisible by 21, then for all $k \in \N$,
		\begin{equation*}
			d_{a^k - 1} = d_{a-1}.
		\end{equation*}
	\end{conjecture}
	Goshima and Yamagishi conjectured a similar statement on tori instead of grids and for $a$ prime \cite{GOSHIMAYAMAGISHI2010}.
	
	\begin{theorem}
		The case of $a=3$ for Conjecture \ref{conj-powers} and \ref{conj-all-2} are equivalent.
	\end{theorem}
	\begin{proof}
		For $a=3$, Conjecture \ref{conj-powers} says that for all $k \in \N$,
		\begin{equation*}
			d_{3^k - 1} = d_{3-1} = 0.
		\end{equation*}
		Since $3^k$ is divisible by 3, Theorem \ref{when-is-deltan-2} tells us that $\delta_{3^k - 1} = 2$.
		So, applying Theorem \ref{sutners-thm},
		\begin{equation*}
			d_{2\cdot3^{k} - 1} = 2d_{3^k - 1} + \delta_{3^k - 1} = 2,
		\end{equation*}
		exactly what Conjecture \ref{conj-all-2} states.
		One can apply all same steps in reverse to shows that Conjecture \ref{conj-all-2} implies the $a=3$ case of Conjecture \ref{conj-powers}.
	\end{proof}
	
	\newpage
	\bibliography{refs.bib}
	\bibliographystyle{amsplain}
\end{document}